\documentclass{amsart}
\usepackage{a4,amssymb,amsmath}

\numberwithin{equation}{section}

\newtheorem{thm}{Theorem}[section]
\newtheorem{lem}[thm]{Lemma}
\newtheorem{prop}[thm]{Proposition}

\theoremstyle{definition}

\theoremstyle{remark}
 \newtheorem{rmk}[thm]{Remark}

\newcommand{\R}{\mathbb{R}}
\newcommand{\N}{\mathbb{N}}

\newcommand{\V}{\mathcal V}
\newcommand{\Z}{\mathcal Z}

\newcommand{\diam}{\operatorname{diam}}
\newcommand{\sq}{\mathcal{S}_d^Q}
\newcommand{\hq}{\mathcal{H}_d^Q}
\newcommand{\sn}{\mathcal{S}_d^n}
\newcommand{\hn}{\mathcal{H}_d^n}
\newcommand{\hqi}{\mathcal{H}_{d_\infty}^Q}
\newcommand{\sqi}{\mathcal{S}_{d_\infty}^Q}
\newcommand{\G}{\mathcal{G}}

\newcommand{\spa}{\operatorname{span}}
\newcommand{\p}{\partial}
\newcommand{\scal}[2]{\langle{#1},{#2}\rangle}
\newcommand{\cut}{\operatorname{\mathcal{C}}}
\newcommand{\F}{\mathcal{F}}
\newcommand{\End}{\operatorname{End}}
\newcommand{\I}{\operatorname{Id}}
\newcommand{\eps}{\varepsilon}
\newcommand{\Hn}{{\mathbb{H}^n}}

\newcommand{\Rn}{\mathbb{R}^{n}}

\newcommand{\leb}{\mathcal{L}^{2n+1}}
\newcommand{\lnn}{\mathcal{L}^{2n}}
\newcommand{\lun}{\mathcal{L}^{1}}
\newcommand{\Rnn}{\mathbb{R}^{2n}}
\newcommand{\vp}{\varphi}

\begin{document}

\title{Isodiametric inequality in Carnot groups}

\author[S\'everine Rigot]{S\'everine Rigot}
\address[S\'everine Rigot]{Universit\'e de Nice Sophia-Antipolis, Laboratoire J.-A. Dieudonn\'e, CNRS-UMR 6621, Parc Valrose, 06108 Nice cedex 02, France}
\email{rigot@unice.fr}




\begin{abstract} The classical isodiametric inequality in the Euclidean space says that balls maximize the volume among all sets with a given diameter. We consider in this paper the case of Carnot groups. We prove that for any Carnot group equipped with a Haar measure one can find a homogeneous distance for which this fails to hold. We also consider Carnot-Carath\'eodory distances and prove that this also fails for these distances as soon as there are length minimizing curves that stop to be minimizing in finite time. Next we study some connections with the comparison between Hausdorff and spherical Hausdorff measures, rectifiability and the generalized 1/2-Besicovitch conjecture giving in particular some cases where this conjecture fails.
\end{abstract}

\maketitle

\section{Introduction}
\label{sec:intro}

The classical isodiametric inequality in the Euclidean space $\R^n$ says that balls maximize the volume among all sets with prescribed diameter,
\begin{equation*}
 \mathcal{L}^n (A) \leq 2^{-n} \alpha_n (\diam A)^n~,~\quad A\subset\R^n~,
\end{equation*}
where $\mathcal{L}^n$ denotes the - $n$-dimensional - Lebesgue measure and $\alpha_n$ the Lebesgue measure of the unit ball in $\R^n$. This gives non trivial information about the geometry of the Euclidean space and about the relation between the Euclidean metric and the Lebesgue measure. In particular a well-known consequence of the isodiametric inequality is the following relation between the $n$-dimensional Lebesgue measure and the $n$-dimensional Haudorff measure defined with respect to the Euclidean distance,
\begin{equation*}
 \mathcal{L}^n =  2^{-n} \alpha_n \mathcal{H}^n~,
\end{equation*}
where $\mathcal{H}^n(A) = \lim_{\delta\downarrow 0} \inf \left\{ \sum_i (\diam A_i)^n~;~A\subset \cup_i A_i~,~\diam A_i\leq \delta\right\}$.

\smallskip
More generally when working on analysis on metric spaces it is a natural question to ask in what kind of settings one can have some - and what - type of isodiametric inequalities and what kind of properties can be deduced from this information. 

\smallskip
In the present paper we are interested in Carnot groups equipped with homogeneous distances. A Carnot group is a connected, simply connected and stratified Lie group. We denote by $Q$ its homogeneous dimension. It can be equipped with a natural family of dilations. A distance $d$ on a Carnot group $G$ is called homogeneous if it induces the topology of the group, is left invariant and one-homogeneous with respect to the dilations. See Section~\ref{sec:carnotgroup} for the definitions. The simplest example of a non trivial Carnot group are the Heisenberg groups. From a metric point of view, natural measures on $(G,d)$ are given by the $Q$-dimensional Hausdorff and spherical Hausdorff measures. They turn moreover out to be non trivial left invariant measures, hence Haar measures of the group, and $Q$-homogeneous with respect to the dilations, so that they are also well-behaved and natural measures to be considered with respect to the structures the space is endowed with. 

\smallskip
We will use throughout this paper the following general definitions and conventions about Hausdorff and spherical Hausdorff measures. Given a metric space $(M,d)$ and a non negative number $n$, we denote by $\hn$ the $n$-dimensional Hausdorff measure given by
\begin{equation*}
 \hn(A)=\lim_{\delta\downarrow 0}~\inf\{~\sum_i (\diam A_i)^n~;~A\subset \cup_i A_i~,~\diam A_i\leq \delta~\},
\end{equation*}
and by $\sn$ the $n$-dimensional spherical Hausdorff measure,
\begin{equation*}
 \sn(A)=\lim_{\delta\downarrow 0}~\inf\{~\sum_i (\diam B_i)^n~;~A\subset \cup_i B_i~,~ B_i \text{ ball}~,~\diam B_i\leq \delta~\}.
\end{equation*}

For a Carnot group equipped with a homogeneous distance $(G,d)$ we will more specifically take as a choice of a reference measure the $Q$-dimensional spherical Hausdorff measure $\sq$. A first fact supporting our choice is that 
\begin{equation*} \label{e:measball}
 \sq(B) = (\diam B)^Q \text{ for any metric ball } B\text{ in } (G,d)~,
\end{equation*}
see Proposition~\ref{prop:sq}. Moreover it turns out that the isodiametric problem can be rephrased in simple terms as we shall explain now.

\smallskip
Going back to the isodiametric problem one seeks for the maximal possible value of the measure of sets with a given diameter. Using dilations this can be rephrased in our context into the problem of finding the maximal possible value of the ratio $\sq(A)/(\diam A)^Q$ among all subsets $A\subset G$ with positive and finite diameter. We denote by $C_d$ this value and call it the isodiametric constant,
\begin{equation} \label{e:cd}
 C_d = \sup\{~\dfrac{\sq(A)}{(\diam A)^Q}~;~0<\diam A <+\infty~\}.
\end{equation}
It obviously follows that one has the following inequality,
\begin{equation*}
\sq(A) \leq C_d~(\diam A)^Q~,~\quad A\subset G~,
\end{equation*}
where $0<C_d<+\infty$ is by definition the best possible constant in the right-hand side. Since $\sq(B) = (\diam B)^Q$ for balls $B$, one has actually $C_d\geq 1$. Hence the best possible inequality one can expect is the following,
\begin{equation} \label{e:ii} \tag{SII}
 \sq(A) \leq (\diam A)^Q~,~\quad A\subset G~,
\end{equation}
that we call the sharp isodiametric inequality. 

\smallskip
We are interested in this paper in the question to know whether the sharp isodiametric inequality holds, or equivalently if balls realize the supremum in the right-hand side of \eqref{e:cd}, in Carnot groups equipped with some specific homogeneous distances. Note that it is not difficult to see that in our context one can always find sets realizing the supremum in \eqref{e:cd}. We shall prove this fact for sake completeness, see Theorem~\ref{thm:existence}. The isodiametric problem can thus be pursued with the question of finding explicitly which are these sets and this will be explored in forthcoming works (see~\cite{lrv}). 

\smallskip
In the present paper we first prove that given a non trivial Carnot group $G$ one can always find some homogeneous distance on $G$, namely $d_\infty$-distances (see~\ref{ssubsec:dinfty}), for which the sharp isodiametric inequality does not hold, see Theorem~\ref{thm:dinfty}. For the specific class of $H$-type groups equipped with the gauge distance (see~\ref{ssubsec:typeH}) we prove with a similar argument that \eqref{e:ii} does not hold either, see Theorem~\ref{thm:dg}. Going back to general Carnot groups, another kind of natural homogeneous distances are the so-called Carnot-Carath\'eodory (or sub-Riemannian) distances (see~\ref{ssubsec:dc}). Equipped with such a distance a Carnot group becomes a geodesic space. We prove that the sharp isodiametric inequality also fails in that case provided there are length minimizing curves that stop to be minimizing in finite time, see Theorem~\ref{thm:dc}. This property holds true in particular in the Heisenberg groups and more generally in $H$-type groups. 

\smallskip
We will then investigate some consequences of these results. One of our initial motivations in the study of the isodiametric problem was indeed some   connections with geometric measure theoretic properties of the space and in particular rectifiability.

\smallskip
First we explicitly state and prove that in our context the comparison between the $Q$-dimensional Hausdorff and spherical Hausdorff measures can be related to the isodiametric problem in the following way,
\begin{equation*}
 \sq = C_d~\hq~,
\end{equation*}
see Proposition~\ref{prop:sqhq}. It follows that 
\begin{equation*}
 \eqref{e:ii} \text{ holds in } (G,d)~\Longleftrightarrow~\hq = \sq~.
\end{equation*}
As a consequence we immediately get that in the cases we are more specifically interested in and where \eqref{e:ii} does not hold, the measures $\hq$ and $\sq$ do not coincide.

\smallskip 
As another consequence of the non-validity of the sharp isodiametric inequality for at least one homogeneous distance, we recover the fact that any Carnot group equipped with a homogeneous distance and with homogeneous dimension $Q$ is purely $Q$-unrectifiable, see Subsection~\ref{subsec:nonrect}. This was already known, see e.g.~\cite{ambkirc},~\cite{magnani}. Existing proofs of this fact, and more generally of related facts, basically involve some algebraic properties of some subgroups of a Carnot group. The point here is that we give a different proof that uses only purely metric arguments. 

\smallskip 
Finally we investigate some connections with the Besicovitch 1/2-problem. First we will note that the density constant $\sigma_Q(G,d)$ (see Subsection~\ref{subsec:bes} for the definition) of a non trivial Carnot group $G$ with homogeneous dimension $Q$ and equipped with a homogeneous distance $d$ satisfies $\sigma_Q(G,d)=C_d^{-1}$, see Theorem~\ref{thm:densconst}. In particular $\sigma_Q(G,d)<1$ whenever the sharp isodiametric inequality does not hold in $(G,d)$. The validity of the bound $\sigma_n(M,d)\leq 1/2$ for any separable metric space $(M,d)$, which was conjectured long ago by A.S. Besicovitch for the one-dimensional density constant in $\R^2$ (see~\cite{besi2}), is known as the generalized Besicovitch 1/2-problem. We give here some counterexamples to this conjecture in the case of the Heisenberg groups equipped with their natural $d_\infty$-distance as well as with the  Carnot-Carath\'eodory distance, see Theorem~\ref{prop:1/2bes-dinfty} and Theorem~\ref{prop:1/2bes-dc}. 

\smallskip 
The paper is organized as follows. In Section~\ref{sec:carnotgroup} we recall the definition of a Carnot group and we introduce the homogeneous distances on these groups we are interested in. Next in this section  we state and prove some general facts about $Q$-dimensional Hausdorff and spherical Hausdorff measures on these groups equipped with homogeneous distances as mentionned before. Section~\ref{sec:isodiam} is devoted to the isodiametric problem itself. We give in this section the proof that the sharp isodiametric inequality does not hold in the cases mentionned above. Finally we investigate in section~\ref{sec:cons} the consequences of this fact described at the end of this introduction.

\section{Carnot groups}
\label{sec:carnotgroup}

\subsection{Carnot groups}
\label{subsec:carnotgroup}
A Carnot group $G$ is a connected and simply connected Lie group whose Lie algebra $\G$ admits a
stratification, 
\begin{equation*}
\G = \oplus_{j=1}^k V_j~,~[V_1,V_j] = V_{j+1}~,~V_k\not= \{0\}~,~V_{k+1} = \{0\}~,
\end{equation*}
for some integer $k\geq 1$ called the step of the stratification.

The exponential map $\exp:\G \rightarrow G$ is then a global diffeomorphism and the group
law is given by the Campbell-Hausdorff formula, 
\begin{equation*}
\exp X \cdot \exp Y = \exp H(X,Y)~,
\end{equation*}
where $H(X,Y) = X+Y+[X,Y]/2+\cdots$ where the dots indicate terms of order $\geq 3$ (the exact formula for $H$ may be found in \cite[II.6.4]{bourbaki}). We will denote by $0$ the unit element of the group. 

We denote by 
\begin{equation*}
 Q=\sum_{j=1}^k j\dim V_j
\end{equation*}
the homogeneous dimension of $G$.

A natural family of dilations on $\G$ is given by $\delta_\lambda(\sum_{j=1}^k Y_j) = \sum_{j=1}^k \lambda^j\,Y_j$, $Y_j\in V_j$, $\lambda>0$. The maps $\exp\circ\delta_\lambda\circ\exp^{-1}$ are group automorphisms of $G$. We shall denote them also by $\delta_\lambda$ and call them dilations on $G$. 

We refer to \cite{fs} for a more detailed presentation of Carnot, and more generally homogeneous, groups.

\subsection{Homogeneous distances}
\label{subsec:homodist}
A distance $d$ on a Carnot group $G$ is called homogeneous if it induces the topology of the group, is left invariant,
\begin{equation*}
 d(x\cdot y,x\cdot z)=d(y,z)
\end{equation*}
for all $x$, $y$, $z \in G$, and one-homogeneous with respect to the dilations, 
\begin{equation*}
 d(\delta_\lambda(y),\delta_\lambda(z))=\lambda\,d(y,z)
\end{equation*}
for all $y$, $z \in G$ and $\lambda>0$. 

A Carnot group equipped with a homogeneous distance is a separable and complete metric space in which closed bounded sets are compact. We also explicitly note that in this context the diameter of a ball is given by twice its radius. There are many ways to define homogeneous distances on a Carnot group. We shall see some examples below. As a general fact we also mention that any two homogeneous distances on a Carnot group are bilipschitz equivalent.

\subsubsection{$d_\infty$-distances}
\label{ssubsec:dinfty}
The first kind of homogeneous distances we will consider in this paper is the class of so-called $d_\infty$-distances. They can be defined in the following way. First one chooses a basis $(X_1,\dots,X_n)$ of $\G$ adapted to the stratification, i.e., $(X_{h_{j-1}+1},\dots,X_{h_j})$ is a basis of $V_j$ where $h_0=0$ and $h_j=\dim V_1+\dots+\dim V_j$. Next one defines an Euclidean norm $\|\cdot\|$ on
$\G$ by declaring the $X_j$'s orthonormal. Then one fix some
positive coefficients $c_j$ so that $\|H(Y,Z) \|_\infty \leq \|Y\|_\infty +
\|Z\|_\infty$ where $\|Y\|_\infty =  \max_j c_j\|Y_j\|^{1/j}$ whenever
$Y=Y_1+\dots+Y_k$, $Y_j\in V_j$. It turns out that one can always find such a family of coefficients (see e.g.~\cite{fssc}). We then set
\begin{equation*}
 \|x\|_\infty = \|\exp^{-1}x\|_\infty \quad \text{and} \quad d_\infty(x,y) = \|x^{-1} \cdot y\|_\infty.
\end{equation*}
It is easy to check that $d_\infty$ is a homogeneous distance on $G$. 

\subsubsection{Carnot-Carath\'eodory distances}
\label{ssubsec:dc}
The second kind of homogeneous distances we wish to consider are the so-called Carnot-Carath\'eodory, also known as sub-Riemannian, distances which we denote by $d_c$. We fix a left invariant Riemannian metric $g$ on $G$ and we set 
\begin{equation*}
 d_c(x,y) = \inf \{ length_{g} (\gamma); \; \gamma \text{ horizontal curve joining } x \text{ to } y\},
\end{equation*}
where a curve is said to be horizontal if it is absolutely continuous and such that at a.e.~every point, its tangent vector belongs to the so-called horizontal subbundle of the tangent bundle whose fiber at some point $x$ is given by $\spa \{X(x);~X \in V_1\}$ when identifying elements in $V_1$ with left invariant vector fields. More generally one actually only needs a scalar product on the first layer $V_1$ of the stratification to define the related Carnot-Carath\'eodory distance in a similar way. 

Recall that by Chow's theorem any two points can be joined by a horizontal curve and $d_c$ turns indeed out to be distance. One can also easily check that it is homogeneous.

An important feature of Carnot-Carath\'eodory distances that makes them natural distances from the geometric point of view is that, endowed with such a distance, a Carnot group becomes a geodesic space, i.e., for all $x$, $y\in G$, there exists a - so-called length minimizing - curve $\gamma \in C([a,b],G)$ such that $\gamma(a)=x$, $\gamma(b)=y$ and $d_c(x,y) = l_{d_c}(\gamma)$ where 
\begin{equation*}
 l_{d_c}(\gamma) = \sup_{N\in  \N^*} \sup_{a=t_0\leq \dots \leq t_N = b} \sum_{i=0}^{N-1} d_c(\gamma(t_i),\gamma(t_{i+1})).
\end{equation*}
Up to a suitable reparameterization, one can - and we will henceforth always assume that - rectifiable curves $\gamma\in C([a,b],G)$, i.e. curves with $l_{d_c}(\gamma)<+\infty$, have constant speed, i.e., $(b-a) \, l_{d_c}(\gamma_{|[s,s']}) = (s'-s)\, l_{d_c}(\gamma)$ for all $s<s' \in [a,b]$. 

For $x\in G$, we denote by $\cut(x)$ the set of points $y\in G$, $y\not= x$, such that one can find a length minimizing curve
$\gamma:[a,b]\rightarrow G$ from $x$ to $y$ which is no more length minimizing after reaching $y$. In other words, if $T>0$ and
$c:[a,b+T]\rightarrow G$ is a rectifiable curve such that $c_{|[a,b]}=\gamma$ then $c$ is not length minimizing on $[a,b+T]$. We note that due to the left invariance of the distance $d_c$, if $\cut(x) \not= \emptyset$ for some $x \in G$ then $\cut(x) \not= \emptyset$ for all $x \in G$. 

We say that $(G,d_c)$ satifies assumption~\eqref{e:cut} if
\begin{equation} \label{e:cut} \tag{$\mathcal{C}$}
 \cut(x) \not= \emptyset \text{ for some, and hence all, } x\in G.
\end{equation}

Note also that due to the homogeneity of $d_c$, if $(G,d_c)$ satifies assumption \eqref{e:cut}, then for all $x\in G$ and $\rho>0$, one can find $y\in \cut(x)$ such that $d_c(x,y)=\rho$.

Heisenberg groups (see e.g. Section~\ref{sec:cons} for the description of one model for these groups) and more generally $H$-type groups (see below for the definition) satisfy assumption~\eqref{e:cut} (see e.g.~\cite{gaveau},~\cite{ambrig}, resp.~\cite{koranyi},~\cite{rigot}, for an explicit description of length minimizing curves in Heisenberg, resp. $H$-type, groups). For general Carnot groups it is to our knowledge a delicate question to know whether assumption~\eqref{e:cut} is satisfied or not. This question is in particular related to delicate issues about length minimizing curves in sub-Riemannian geometry. 

We will prove in this paper that the sharp isodiametric inequality does not hold in $(G,d_c)$ as soon as $(G,d_c)$ satifies assumption~\eqref{e:cut}.

\subsubsection{$H$-type groups and gauge distance.}
\label{ssubsec:typeH}
We recall here the definition of $H$-type groups. These Carnot groups are classical groups concerning many aspects of analysis, in particular Harmonic Analysis and PDE's, on homogeneous groups.  We will prove the non validity of the sharp isodiametric inequality for that class of groups equipped with the gauge distance. 

Let $\G$ be a Lie algebra equipped with a scalar product and which can be decomposed in a non trivial orthogonal direct sum, $\G = \V \oplus \Z$, where $[\V,\V]\subset\Z$ and $[\V,\Z]=[\Z,\Z]=\{0\}$. We define the linear map $J:\Z \rightarrow \End \V$ by $\scal{J(Z)X}{X'} = \scal{Z}{[X,X']}$ for all $X$, $X'\in\V$ and $Z\in\Z$. We say that $\G$ is a $H$-type algebra if for all $Z\in\Z$,
\begin{equation*} 
J(Z)^2 = -\|Z\|^2\,\I~.
\end{equation*}
Note that in such a case $[\V,\V]=\Z$ so that $\G$ is a stratified Lie algebra of step two. For more details about $H$-type algebras we refer to e.g.~\cite{Ka}. A Carnot group is said to be a $H$-type group if its Lie algebra is of $H$-type.

This class of groups gives a natural generalization of the Heisenberg groups which correspond to the case where $\dim \Z = 1$. 

The gauge distance $d_g$ on a $H$-type group $G$ is defined in the following way. One defines a homogeneous norm on $G$ by 
\begin{equation*}
\|\exp(X+Z)\|_g = (\|X\|^4 + 16 \,\|Z\|^2)^{1/4}
\end{equation*}
for all $X\in\V$ and $Z\in\Z$ where the norm in the right-hand side is the norm induced by the given scaler product on $\G$. By \cite{Cy}, this norm satisfies $\|x \cdot y\|_g\leq \|x\|_g + \|y\|_g$ for all $x$, $y\in G$, and one then defines the gauge distance between any two points $x$, $y\in G$ by
\begin{equation*}
d_g(x,y) = \|x^{-1}\cdot y\|_g~.
\end{equation*}
which is clearly homegeneous.

\subsection{Hausdorff measures}
\label{sec:hausdorffmeas}
As a classical fact the Hausdorff dimension of a Carnot group equipped with a homogeneous distance $(G,d)$ coincides with its homogeneous dimension $Q$. Moreover both measures $\hq$ and $\sq$ (see Section~\ref{sec:intro} for the definitions and conventions about Hausdorff measures) are left invariant measures which give positive and finite measure to any ball with positive and finite diameter and hence are Haar measures of the group. These measures are also $Q$-homogeneous with respect to the dilations of the group,
\begin{equation*}
 \hq(\delta_\lambda (A)) = \lambda^Q~\hq(A)~, \quad \sq(\delta_\lambda (A)) = \lambda^Q~\hq(A)~,
\end{equation*}
for all $A\subset G$ and $\lambda>0$. As a consequence any Haar measure $\mu$ of the group is $Q$-homogeneous as well and in particular we have $\mu(B) = C~(\diam B)^Q$ for any ball $B$ in $(G,d)$ and for some constant $C>0$ which depends only on the homogeneous distance $d$ and the Haar measure $\mu$. As already mentionned in the introduction we will take as a reference measure in $(G,d)$ the measure $\sq$. The first observation supporting this choice is the following fact which gives in our context the exact - and particularly simple - value of the $\sq$-measure of a ball.

\begin{prop} \label{prop:sq}
 Let $(G,d)$ be a Carnot group equipped with a homogeneous distance. Let $Q$ denote its homogeneous dimension. Then
\begin{equation*}
 \sq(B) = (\diam B)^Q
\end{equation*}
for any ball $B$ in $(G,d)$.
\end{prop}

This follows from classical covering arguments and might be well-known as well although to our knowledge not explicitly stated that way in the literature. We give a proof below for sake of completeness. We first recall the general covering arguments to be used. Let $\F$ be a family of subsets of $G$. We say that $\F$ covers a set $A$ finely if for all $a\in A$ and all $\eps>0$, there exists $F\in\F$ such that $a\in F$ and $\diam F \leq \eps$. We say that $\F$ is adequate for a set $A$ if for any open set $U$ there exists a countable disjointed subfamily $\tilde\F$ of $\F$ such that $\cup_{F\in\tilde\F} F \subset U$ and $\mu((U\cap A)\setminus \cup_{F\in\tilde\F} F) = 0$ where $\mu$ denotes any Haar measure of the group. If $F\in\F$, we set $\hat F = \cup \{T\in\F~;~T\cap F \not = \emptyset~,~\diam (T) \leq 2\,\diam (F)\}$.

\begin{thm} \cite[Chapter 2.8]{Fed} 
\label{thm:covering}
Assume that all sets in $\F$ are closed and bounded and that $\mu(F)>0$ for all $F\in\F$. If $\F$ covers a set $A$ finely and if there exists $\tau\geq 1$ such that $\mu(\hat F) \leq \tau \mu (F)$ for all $F\in\F$, then $\F$ is adequate for $A$.
\label{cov}\end{thm} 

\noindent\textit{Proof of Proposition~\ref{prop:sq}.}
Let $\nu$ denote the Haar measure of the group normalized in such a way that $\nu(B)= (\diam B)^Q$ for any ball $B$ in $(G,d)$. It will actually follow from the proposition that $\sq=\nu$. Let $B$ be a ball in $(G,d)$. Let $\eps>0$ be fixed and $(B_i)_i$ be a countable family of balls such that $B\subset \cup_i B_i$ and $\sum_i (\diam B_i)^Q \leq \sq(B) +\eps$. We have
\begin{equation*}
 \nu(B) \leq \sum_i \nu(B_i) = \sum_i (\diam B_i)^Q \leq \sq(B) +\eps.
\end{equation*}
Since this is true for all $\eps>0$, we get that $(\diam B)^Q=\nu(B)\leq \sq(B)$. 

Conversely let $\delta>0$ be fixed and $\F$ denote the family of closed balls in $(G,d)$ with positive diameter less than $\delta$. Let $B$ be an open ball. Then $\F$ covers $B$ finely and satisfies the assumptions of Theorem~\ref{thm:covering}. Hence $\F$ is adequate for $B$ and one can find a countable disjointed subfamily $\tilde \F \subset \F$ such that $\cup_{F\in \tilde \F} F \subset B$ and $\sq(B\setminus \cup_{F\in \tilde \F} F) = 0$. It follows that
\begin{equation*}
 \mathcal{S}^Q_{d,\delta} (\cup_{F\in \tilde \F} F) \leq \sum_{F\in \tilde \F} (\diam F)^Q = \nu(\cup_{F\in \tilde \F} F) \leq \nu(B)
\end{equation*}
where $\mathcal{S}^Q_{d,\delta} (A)= \inf \{ \sum_i (\diam B_i)^Q~;~A\subset \cup_i B_i~,~ B_i \text{ ball}~,~\diam B_i\leq \delta\}$ for any $A\subset G$. We have $\mathcal{S}^Q_{d,\delta}(B) \leq \mathcal{S}^Q_{d,\delta}(\cup_{F\in \tilde \F} F) + \mathcal{S}^Q_{d,\delta}(B\setminus \cup_{F\in \tilde \F} F)$ and $\mathcal{S}^Q_{d,\delta}(B\setminus \cup_{F\in \tilde \F} F) \leq \sq(B\setminus \cup_{F\in \tilde \F} F) = 0$. It follows that 
\begin{equation*}
 \mathcal{S}^Q_{d,\delta}(B) \leq \nu(B).
\end{equation*}
Letting $\delta\downarrow 0$, we get that $\sq(B) \leq \nu(B)$. 

Hence the claim follows for any open, and then also automatically for any closed, ball. \hfill $\Box$
\medskip

Next we note and explicitly prove that the ratio between $\hq$ and $\sq$ can be related to the isodiametric constant $C_d$ (see~\eqref{e:cd} for the definition of $C_d$).

\begin{prop} \label{prop:sqhq}
 Let $(G,d)$ be a Carnot group equipped with a homogeneous distance. Let $Q$ denote its homogeneous dimension. Then
\begin{equation*}
\sq = C_d~\hq~.
\end{equation*}
\end{prop}

\begin{proof}
 The proof is very similar to the proof of of Proposition~\ref{prop:sq}. First we note that it is enough to prove that $\sq(U) = C_d~\hq(U)$ for any open set $U$ with finite Haar measure. 

Let $\eps>0$ be fixed and $(U_i)_i$ be a countable family of bounded sets such that $U\subset \cup_i U_i$ and $\sum_i (\diam U_i)^Q \leq \hq(U) +\eps$. We have
$$ \sq(U) \leq \sum_i \sq(U_i)  \leq \sum_i C_d~(\diam U_i)^Q \leq C_d~(\hq(U) +\eps).$$
Since this is true for all $\eps>0$, we get that $ \sq(U) \leq C_d~\hq(U)$. 

Conversely for any subset $A$ of $G$ with $0< \diam A <+\infty$ we set $C(A)=\sq(A)/(\diam A)^Q$. We prove that $C(A)~\hq(U)\leq \sq(U)$. The required inequality $C_d~\hq(U)\leq \sq(U)$ will then follow from the definition of $C_d$. First we note that $C(A) \leq C(\overline A)$ and $C(A_{x,\lambda})
= C(A)$ where $A_{x,\lambda} =\delta_\lambda (x \cdot A)$ with $x\in G$ and $\lambda>0$. Hence, taking closure and up to a translation, one can assume with no
loss of generality that $A$ is closed and $0\in A$. We also assume that $C(A)>0$ otherwise there is nothing to prove. Let $\delta>0$ be fixed and set $\F=\{A_{x,\lambda}~;~x\in G~,~\lambda\leq \delta/ \diam A\}$. We have $x\in A_{x,\lambda}$ and $\diam A_{x,\lambda} = \lambda\diam A$ for any $x\in G$ and $\lambda>0$ hence $\F$ covers $U$ finely. Next we have $\diam\hat A_{x,\lambda} \leq 5\diam A_{x,\lambda}$ hence $\hat A_{x,\lambda}$ is contained in a ball with diameter $10\diam A_{x,\lambda}$. On the other hand it follows from Proposition~\ref{prop:sq} that $C(A) = \mu(A)/\mu(B)$ where $B$ is a ball with diameter $\diam A$ and $\mu$ any Haar measure. Then we get that $\mu(\hat A_{x,\lambda}) \leq 10^Q \mu (A_{x,\lambda}) / C(A_{x,\lambda}) = 10^Q \mu (A_{x,\lambda}) / C(A)$ for any Haar measure $\mu$. The family $\F$ thus satisfies the assumptions of Theorem~\ref{thm:covering}. Hence $\F$ is adequate for $U$ and one can find a countable disjointed subfamily $\tilde \F \subset \F$ such that $\cup_{F\in \tilde \F} F \subset U$ and $\hq(U\setminus \cup_{F\in \tilde \F} F) = 0$. Arguing as in the proof of Proposition~\ref{prop:sq} it follows that 
\begin{multline*}
 C(A)~\mathcal{H}^Q_{d,\delta} (U) = C(A)~\mathcal{H}^Q_{d,\delta} (\cup_{F\in \tilde \F} F) \leq C(A)~\sum_{F\in \tilde \F} (\diam F)^Q\\ = \sum_{F\in \tilde \F} C(F)~(\diam F)^Q  = \sum_{F\in \tilde \F} \sq(F) = \sq(\cup_{F\in \tilde \F} F) \leq \sq(U)
\end{multline*}
where $\mathcal{H}^Q_{d,\delta} (E)= \inf \{ \sum_i (\diam E_i)^Q~;~E\subset \cup_i E_i~,~\diam E_i\leq \delta\}$ for any $E\subset G$. Letting $\delta\downarrow 0$, we get that $C(A)~\hq(U) \leq \sq(U)$ as wanted.
\end{proof}

\section{Isodiametric problem}
\label{sec:isodiam}

This section is devoted to the isodiametric problem in Carnot groups equipped with homogeneous distances. We refer to the introduction (see Section~\ref{sec:intro}) for the statement of this problem. We shall in particular follow here the conventions and notations introduced in Section~\ref{sec:intro}. When not explicitly stated it is also always implicitly assumed throughout this section that the notations $G$, resp. $d$, resp. $Q$, denotes a Carnot group, resp. a homogeneous distance on $G$, resp. the homogeneous dimension of $G$.

\smallskip
We say that a set $E$ is isodiametric in $(G,d)$ if $E$ is a compact subset of $G$ which maximizes the ratio $\sq(A)/(\diam A)^Q$ among all $A\subset G$ with positive and finite diameter and hence satisfying 
\begin{equation*}
 \sq(E) = C_d~(\diam E)^Q~.
\end{equation*} 
See~\eqref{e:cd} for the definition of the isodiametric constant $C_d$. Note that it is not restrictive to ask isodiametric sets to be compact as the closure of any set which realizes the supremum in the right-hand side of \eqref{e:cd} is a compact set that still realizes the supremum.

\begin{thm} \label{thm:existence}
 Let $(G,d)$ be a Carnot group equipped with a homogeneous distance. Let $Q$ denote its homogeneous dimension. Then isodiametric sets in $(G,d)$ do exist.
\end{thm}

\begin{proof}Using dilations it is enough to find a compact set $E$ with $\diam E = 1$ and such that $\sq(E) = \sup \{\sq(A)~;~\diam A =1\}$. Let $(E_i)$ be a maximizing sequence of sets with diameter 1 such that 
\begin{equation*}
\lim_{i\rightarrow+\infty} \sq(E_i) = \sup \{\sq(A)~;~\diam A =1\}.
\end{equation*}
Taking closure if necessary and up to a suitable translation one can assume each $E_i$ to be compact and contained in the closed unit ball. Then one can extract a subsequence, still denoted by $(E_i)$, converging in the Hausdorff distance to some compact set $E$.

The measure $\sq$ being a regular measure it follows that $\sq$ is upper semi-continuous with respect to the convergence of compact sets in the Hausdorff distance. Indeed let $(A_i)$ be a sequence of compact sets converging to some compact set $A$ in the Hausdorff distance. By regularity of $\sq$, for any $\eps>0$, one can find an open set $U\supset A$ such that $\sq(U) \leq \sq(A) + \eps$. On the other hand by convergence in the Hausdorff distance we have $A_i\subset U$ for all $i$ large enough and hence $\sq(A_i)\leq \sq(U)$. It follows that 
\begin{equation*}
 \limsup _{i\rightarrow+\infty} \sq(A_i) \leq \sq(U) \leq \sq(A) + \eps~.
\end{equation*}
Since this holds for any $\eps>0$ we finally get 
\begin{equation*}
 \limsup _{i\rightarrow+\infty} \sq(A_i) \leq \sq(A)
\end{equation*}
as wanted.

Going back to the maximizing sequence $(E_i)$, it follows that 
\begin{equation*}
 \sq(E) \geq \lim_{i\rightarrow+\infty} \sq(E_i) = \sup \{\sq(A)~;~\diam A =1\}.
\end{equation*}
On the other hand, by convergence in the Hausdorff distance, we have $\diam E = 1$ and hence $E$ is isodiametric.
\end{proof}

Next we give a sufficient condition for a set to be isodiametric.

\begin{lem} \label{nonisodiam} Let $(G,d)$ be a Carnot group equipped with a homogeneous distance. Let $E\subset G$ be a compact set with $\diam E>0$. Assume that one can find $x\in \p E$ such that $d(x,y) <\diam E$ for all $y\in E$. Then $E$ is not isodiametric in $(G,d)$.
\end{lem}

\begin{proof}
 Let $E$ and $x\in \p E$ be as in the statement. Since $E$ is compact we have $\sup_{y\in E} d(x,y) <\diam E$. The map $z\mapsto \sup_{y\in E} d(z,y)$ being 1-Lipschitz and hence continuous, it follows that one can find a closed ball $B$ centered at $x$ with $0<\diam B< \diam E$ such that $\sup_{y\in E} d(z,y) < \diam E$ for all $z\in B$. We set $A=E\cup B$. We have $\diam A = \diam E$. Since $x\in\p E$ and $E$ is closed, we have that $A\setminus E$ has non empty interior and hence positive $\sq$-measure. Hence $\sq(A)>\sq(E)$ which proves that $E$ is not isodiametric.
\end{proof}

\begin{rmk}
 One can actually prove that if $E$ is isodiametric and $x\in \p E$ then any $y\in E$ such that $d(x,y) = \diam E$ belongs to the boundary of $E$. This follows essentially from the fact that the distance function from a given point is in our context an open map (recall also that isodiametric sets are assumed to be compact). We will however not use this fact in this paper.
\end{rmk}

From now on  we will consider non trivial Carnot groups, i.e., Carnot groups with a stratification of step $\geq 2$. When the stratification is trivial, i.e. of step 1, one recovers the abelian group $(\R^n,+)$. Then dilations are given by the multiplication by a scalar factor and homogeneous distances are distances induced by a norm in the classical usual sense. It is well-known that the sharp isodiametric inequality holds in $\Rn$ equipped with any  distance induced by a norm. On the contrary the situation becomes definitely different as soon as the stratification is non trivial as we shall see below.

\smallskip
First we prove that for any non trivial Carnot group, there exists a homogeneous distance, namely the $d_\infty$-distances, for which the sharp isodiametric inequality does not hold.

\begin{thm} \label{thm:dinfty}
Let $(G,d_\infty)$ be a non trivial Carnot group equipped with a homogeneous $d_\infty$-distance. Then closed balls in $(G,d_\infty)$ are not isodiametric and the sharp isodiametric inequality \eqref{e:ii} does not hold in $(G,d_\infty)$.
\end{thm}

\begin{proof}
Let the $d_\infty$-distance be defined as in ~\ref{ssubsec:dinfty}. Using left translations and dilations we only need to prove that the closed unit ball in $(G,d_\infty)$ is not isodiametric. Let $B$ denote the closed unit ball in $(G,d_\infty)$. Let $X\in V_k$ be such that $c_k\|X\|^{1/k}=1$ and set $x=\exp X$. Then $x\in \p B$. Let $y\in B$, $y=\exp (Y_1+\dots+Y_k)$, $Y_j\in V_j$. We have 
\begin{equation*}
 x^{-1} \cdot y = \exp(-X)\cdot \exp (\sum_{j=1}^k Y_j) = \exp(\sum_{j=1}^{k-1} Y_j + Y_k-X).
\end{equation*}
Since 
\begin{equation*}
c_k^{2k}\|Y_k-X\|^2 = c_k^{2k}\|Y_k\|^2+c_k^{2k}\|X\|^2-2c_k^{2k}\scal{Y_k}{X}\leq 2+2c_k^k\|Y_k\|c_k^k\|X\|\leq 4,
\end{equation*}
we get $d_\infty(x,y) \leq \max(1,2^{1/k})<2=\diam B$ and the conclusion follows from Lemma~\ref{nonisodiam}.
\end{proof}

Similarly the sharp isodiametric inequality~\eqref{e:ii} does not hold in $H$-type groups equipped with the gauge distance. See~\ref{ssubsec:typeH} for the definition of $H$-type groups and gauge distance.

\begin{thm} \label{thm:dg}
Let $(G,d_g)$ be a $H$-type group equipped with the gauge distance. Then closed balls in $(G,d_g)$ are not isodiametric and the sharp isodiametric inequality \eqref{e:ii} does not hold in $(G,d_g)$.
\end{thm}

\begin{proof} Arguing as in the proof of Theorem~\ref{thm:dinfty} we prove that the closed unit ball in $(G,d_g)$ - denoted by $B$ - is not isodiametric. Following the notations of~\ref{ssubsec:typeH} we choose some $x=\exp(Z_x)$ with $Z_x\in\Z$ such that $\|Z_x\| = 1/4$. Then if $y = \exp(X+Z)\in B$ with $X\in\V$ and $Z\in\Z$, we have $\exp(Z_x)^{-1} \cdot \exp(X+Z) = \exp(X + Z - Z_x)$ hence
\begin{equation*}
\begin{split}
d_g(x,y)^4 &= \|X\|^4 + 16\,\|Z-Z_x\|^2\\
&= \|X\|^4 +16 \, \|Z\|^2 +16 \, \|Z_x\|^2 - 32\,\scal{Z_x}{Z} \leq 2\,(1+16\,\|Z\|\,\|Z_x\|) \leq 4
\end{split}
\end{equation*}
since $y\in B$ and hence $\|X\|^4 +16 \, \|Z\|^2 \leq 1$ and $\|Z\|\leq 1/4$. It follows that $d_g(x,y)\leq \sqrt 2 <2=\diam B$ and we get the conclusion  from Lemma~\ref{nonisodiam}.
\end{proof}

Finally we investigate the case of the Carnot-Carath\'eodory distances. See~\ref{ssubsec:dc} for the definition of Carnot-Carath\'eodory distances.

\begin{thm} \label{thm:dc}
 Let $(G,d_c)$ be a non trivial Carnot group equipped with a Carnot-Carath\'eodory distance. Assume that $(G,d_c)$ satisfies assumption \eqref{e:cut}.
 Then closed balls in $(G,d_c)$ are not isodiametric and the sharp isodiametric inequality \eqref{e:ii} does not hold in $(G,d_c)$.
\end{thm}

\begin{proof}
 By assumption \eqref{e:cut} one can find $x\in\cut(0)$ with $d_c(0,x)=1$ and $\gamma:[0,1]\rightarrow G$ a length minimizing curve from 0 to $x$ which is no more length minimizing after reaching $x$. Let $B$ denote the closed ball with center $x$ and radius one. We prove that $B$, and hence any closed ball, is not isodiametric in $(G,d_c)$. Let $y\in B$ and $c:[0,1+d_c(x,y)]\rightarrow G$ be such that $c_{|[0,1]}=\gamma$ and $c_{|[1,1+d_c(x,y)]}$ is a unit speed length minimizing curve from $x$ to $y$. Then $c$ is not length minimizing on $[0,1+d_c(x,y)]$ by choice of $\gamma$ hence $d_c(0,y)<l_{d_c}(c) = d_c(0,x) + d_c(x,y)\leq 2 = \diam B$ and the conclusion follows from Lemma \ref{nonisodiam}.
\end{proof}

\section{Some consequences and related problems}
\label{sec:cons}

\subsection{Non-rectifiabilty}
\label{subsec:nonrect}

We recover here the fact that a non trivial Carnot group equipped with a homogeneous distance and with homogeneous dimension $Q$ is purely $Q$-unrectifiable. This is already well-known, see e.g.~\cite{ambkirc},~\cite{magnani} for more and further results about rectifiability in Carnot groups. The point is that we give here a different proof using only metric arguments. Recall that a metric space $(M,d)$ is said to be purely $n$-unrectifiable if one has $\hn(f(A)) = 0$ for every Lipschitz map $f:A\subset\R^n \rightarrow (M,d)$ where $\R^n$ is equipped with the Euclidean distance.

\begin{thm} \label{thm:nonrect}
 Let $(G,d)$ be a non trivial Carnot group equipped with a homogeneous distance. Let $Q$ denote its homogeneous dimension. Then $(G,d)$ is purely $Q$-unrectifiable.
\end{thm}

\begin{proof}
 The property of being purely $Q$-unrectifiable is invariant under a change of bilipschitz equivalent distance. Since all homogeneous distances are bilipschitz equivalent, one can assume that $G$ is equipped with a $d_\infty$-distance defined as in~\ref{ssubsec:dinfty}. Assume by contradiction that one can find a Lipschitz map $f:A\subset \R^Q \rightarrow (G,d_\infty)$ such that $\hqi(f(A))>0$. One can also assume with no loss of generality that $A$ is bounded and hence $\hqi(f(A))<+\infty$. Then one has
\begin{equation*}
 \lim_{r\downarrow 0} \dfrac{\hqi(f(A) \cap B(x,r))}{(2r)^Q} =1
\end{equation*}
for $\hqi$-a.e. $x\in f(A)$ (see~\cite{kirc}). On the other hand it follows from Proposition~\ref{prop:sq}, Proposition~\ref{prop:sqhq} and Theorem~\ref{thm:dinfty} that 
\begin{equation*}
 \dfrac{\hqi(f(A) \cap B(x,r))}{(2r)^Q} \leq \dfrac{\hqi(B(x,r))}{(2r)^Q} = \dfrac{\hqi(B(x,r))}{\sqi(B(x,r))} = C_{d_\infty}^{-1} <1
\end{equation*}
which gives a contradiction.
\end{proof}

\subsection{Besicovitch 1/2-problem}
\label{subsec:bes}

We investigate now some connection between the isodiametric problem and densities. Following~\cite{preisstiser} given a metric space $(M,d)$ we denote by $\sigma_n(M,d)$ the smallest number such that every subset $A\subset M$ of finite $\hn$-measure having 
\begin{equation*}
\underline{D}_n(A,x) > \sigma_n(M,d)
\end{equation*}
at $\hn$-a.e. $x\in A$ is $n$-rectifiable, where the lower $n$-density of $A$ at $x$ is given by
\begin{equation*}
\underline{D}_n(A,x)= \liminf_{r\downarrow 0}\dfrac{\hn(A\cap B(x,r))}{(2r)^n}.
\end{equation*}
We call this number the density constant. Recall that $A$ is said $n$-rectifiable if $\hn$-almost all of $A$ can be covered by countably many Lipschitzian images of subsets of the Euclidean space $\R^n$. 

\smallskip
It follows from famous results of A.S. Besicovitch (see~\cite{besi1}, \cite{besi2}) that one always has $\sigma_n(M,d) \leq 1$. These results also imply that the fact that $\sigma_n(M,d) = 1$ does not give any positive information about rectifiability. See~\cite{preisstiser} for a more detail account about known results on density constants.

\smallskip
We first note and explicitly prove that $\sigma_Q(G,d)<1$ whenever $d$ is a homogeneous distance on a Carnot group $G$ for which the sharp isodiametric inequality~\eqref{e:ii} does not hold and $Q$ denotes the homogeneous dimension of $G$. More precisely

\begin{thm} \label{thm:densconst}
 Let $(G,d)$ be a non trivial Carnot group equipped with a homogeneous distance. Let $Q$ denote its homogeneous dimension. Then one has 
\begin{equation*}
 \sigma_Q(G,d)= C_d^{-1}
\end{equation*}
and hence $\sigma_Q(G,d)<1$ if and only if the sharp isodiametric inequality~\eqref{e:ii} does not hold in $(G,d)$.
\end{thm}

\begin{proof} Let $U\subset G$ be open. It follows from Proposition~\ref{prop:sq} and Proposition~\ref{prop:sqhq} that
\begin{equation*}
 \underline{D}_n(U,x) = \liminf_{r\downarrow 0}\dfrac{\hq(B(x,r))}{(2r)^Q} = \liminf_{r\downarrow 0}\dfrac{\hq(B(x,r))}{\sq(B(x,r))} = C_d^{-1}
\end{equation*}
for all $x\in U$. Since $(G,d)$ is purely $Q$-unrectifiable (see e.g. Theorem~\ref{thm:nonrect}) it follows that $\sigma_Q(G,d)\geq C_d^{-1}$. On the other hand, one obviously has 
 \begin{equation*}
\underline{D}_n(A,x)\leq \underline{D}_n(G,x) = C_d^{-1}
\end{equation*}
for all $A\subset G$ and $x\in G$. Hence $\sigma_Q(G,d)\leq C_d^{-1}$ which concludes the proof. 
\end{proof}

A.S. Besicovitch conjectured in~\cite{besi2} that $\sigma_1(\R^2)\leq 1/2$. More generally the question to know whether $\sigma_n(M,d) \leq 1/2$ for any separable metric space $(M,d)$ is known as the generalized Besicovitch 1/2-problem. We investigate below the case of the Heisenberg groups equipped with a $d_\infty$-distance as well as with a Carnot-Carath\'eodory distance exhibiting cases where the conjecture fails. 

\smallskip
We first describe the model we consider for the Heisenberg group $\Hn$. We identify it with $\Rnn\times\R$ and denote points in $\Hn$ by $[z,t]$, $z=(x_1,\dots,x_{2n})\in\Rnn$, $t\in\R$. The group law is 
\begin{equation*}
[z,t]\cdot [z',t'] = [z + z', t+t'+ 2 \sum_{j=1}^n (x_{n+j} x'_j - x_j x'_{n+j})]~.
\end{equation*}
The stratification is given by 
\begin{equation*}
 V_1 = \spa \{X_j, Y_j; \, j=1,\dots,n\}~,~V_2 = \spa \{\partial_t\}~,
\end{equation*}
where $X_j = \partial_{x_j} + 2 x_{n+j} \partial_t$ and $Y_j = \partial_{x_{n+j}} - 2 x_j \partial_t$. The dilations are $\delta_\lambda([z,t]) = [\lambda z, \lambda^2 t]$. The homogeneous dimension of $\Hn$ is $2n+2$. The $(2n+1)$-dimensional Lebesgue measure $\leb$ on $\Hn \approx \Rnn\times\R$ is a Haar measure of the group.

\smallskip
We consider first the $d_\infty$-distance defined with respect to the homogeneous norm $\|[z,t]\|_\infty = \max (\|z\|,|t|^{1/2})$ where here $\|\cdot\|$ denotes the Euclidean norm in $\Rnn$, 
\begin{equation*}
 d_\infty([z,t],[z',t'])= \|[z,t]^{-1}\cdot [z',t']\|_\infty~.
\end{equation*}

\begin{thm} \label{prop:1/2bes-dinfty}
 The isodiametric constant $C_{d_\infty}$ in $(\Hn,d_\infty)$ satisfies $C_{d_\infty}<2 $ and hence $\sigma_{2n+2}(\Hn,d_\infty)>1/2$.
\end{thm}

\begin{proof}
First we note that due to Proposition~\ref{prop:sq} and using dilations the isodiametric constant can be rewritten as
\begin{equation*}
C_{d_\infty} = \dfrac{\sup \leb(A)}{\leb(B)}
\end{equation*}
where the supremum is taken over all compact subsets $A$ of $(\Hn,d_\infty)$ with $\diam A = 2$ and $B$ denotes the closed unit ball in $(\Hn,d_\infty)$. We have $B=\{[z,t]\in\Hn~;~\|z\|\leq 1~,~|t|\leq 1~\}$ hence
\begin{equation*}
 \leb(B) = 2\,\lnn(U)
\end{equation*}
where $U$ denotes the closed unit Euclidean ball in $\Rnn$ and $\lnn$ the $2n$-dimensional Lebesgue measure on $\Rnn$. Let $p:\Hn\rightarrow \Rnn$ denote the projection defined by $p([z,t]) = z$. Let $A$ be a compact subset of $(\Hn,d_\infty)$ with $\diam A = 2$. The map $p$ is 1-Lipschitz from $(\Hn,d_\infty)$ to the Euclidean space $(\Rnn,\|\cdot\|)$ hence
\begin{equation} \label{e:measproj}
 \lnn(p(A)) \leq \lnn(U)
\end{equation}
by the classical isodiametric inequality in $(\Rnn,\|\cdot\|)$. Next for any $z\in p(A)$, we have $|t-t'| \leq d_\infty([z,t],[z,t'])^2\leq 4$ whenever $[z,t]$ and $[z,t']\in A$ hence $A\cap p^{-1}(\{z\})$ is contained in some segment with length at most 4. It follows from Fubini's Theorem that 
\begin{equation*}
 \leb(A) \leq 4\,\lnn(U)
\end{equation*}
and hence
\begin{equation*}
 C_{d_\infty} \leq 2.
\end{equation*}

Assume by contradiction that $C_{d_\infty}=2$ and thus (see Theorem~\ref{thm:existence}) that one can find a compact subset $E$ in $(\Hn,d_\infty)$  with $\diam E = 2$ and such that $\leb(E) = 4\,\lnn(U)$. By the argument above we know that for any $z\in p(E)$ we have $\lun(E\cap p^{-1}(\{z\})) \leq 4$ where $\lun$ denotes the one-dimensional Lebesgue measure. On the other hand, remembering~\eqref{e:measproj}, we have
\begin{equation} \label{e:measE}
\int_{p(E)} (\lun(E\cap p^{-1}(\{z\})) -4)\,d\lnn(z) \geq \leb(E) - 4\,\lnn(U) = 0.
\end{equation}
It follows that $\lun(E\cap p^{-1}(\{z\})) = 4$ and consequently that $E\cap p^{-1}(\{z\})$ is a whole segment with length 4, $E\cap p^{-1}(\{z\}) = \{[z,t]~;~ t_z^- \leq t \leq t_z^+~\}$ for some $t_z^-$ and $t_z^+$ such that $t_z^+ - t_z^- = 4$, for $\lnn$-a.e.~$z\in p(E)$. Up to a translation one can assume that $0\in E$ and that $E\cap p^{-1}(\{0\}) = \{[0,t]~;~ 0 \leq t \leq 4~\}$. Then the fact that $d_\infty([0,4],[z,t_z^-])\leq 2$ implies in particular that $t_z^-\geq 0$ and on the other hand $d_\infty([0,0],[z,t_z^+])\leq 2$ implies that $t_z^+ \leq 4$ and we finally get that $t_z^- = 0$ and $t_z^+=4$ for $\lnn$-a.e.~$z\in p(E)$.

For $z=(x_1,\dots,x_{2n})\in\Rnn$ we consider now the half-space $H_z$ in $\Rnn$ defined by $H_z=\{z'=(x'_1,\dots,x'_{2n})\in\Rnn~;~\sum_{j=1}^n (x_{n+j} x'_j - x_j x'_{n+j})<0\}$. We have $z\in\partial H_z$. On the other hand, considering points in $p(E)$ where $p(E)$ has density one, we recall that
\begin{equation*}
 \lim_{r\downarrow 0} \dfrac{\lnn(p(E) \cap U(z,r))}{\lnn(U(z,r))} =1
\end{equation*}
for $\lnn$-a.e. $z\in p(E)$ where $U(z,r)$ denotes the (closed) ball in $(\Rnn,\|\cdot\|)$ with center $z$ and radius $r$. It follows that $\lnn(p(E) \cap H_z)>0$ for $\lnn$-a.e. $z\in p(E)$.

All together we finally get that one can find $z$ and $z' \in p(E)$ such that $t_{z}^- = t_{z'}^- = 0$, $t_{z}^+ = t_{z'}^+ = 4$ and $z'\in H_z$. It follows that 
\begin{equation*}
 \diam E \geq d_\infty([z, 0],[z',4]) \geq (4 - 2\sum_{j=1}^n (x_{n+j} x'_j - x_j x'_{n+j}))^{1/2} > 2
\end{equation*}
which gives a contradiction and concludes the proof.
\end{proof}

We consider now the Carnot-Carth\'eodory distance $d_c$ on $\Hn$ defined in the following way. We consider the left invariant Riemannian metric $g$ on $\Hn$ that makes $(X_1,\dots,X_n,Y_1,\dots,Y_n,\partial_t)$ an orthonormal basis and we define $d_c$ as in~\ref{ssubsec:dc}. Explicit description of balls in $(\Hn,d_c)$ are well-known. We refer to e.g. \cite{gaveau} or \cite{ambrig}. If we denote by $B$ the closed unit ball in $(\Hn,d_c)$ one has 
\begin{equation*}
B = \{~[\dfrac{\sin\vp}{\vp}\cdot\chi~,~\dfrac{2\vp - \sin(2\vp)}{2\vp^2}\,|\chi|^2]~;~\chi\in\Rnn~,~|\chi|\leq 1~,~\vp\in[-\pi,\pi]~\}.
\end{equation*}
It follows that
\begin{equation*}
\begin{split}
 \leb(B) &= 4n\alpha_{2n}\,\int_0^\pi \dfrac{2\vp - \sin(2\vp)}{2\vp^2} \cdot \left(\dfrac{\sin\vp}{\vp}\right)^{2n-1} \cdot \left|\left(\dfrac{\sin\vp}{\vp}\right)'\right| \, d\vp\\
&= 4n\alpha_{2n}\,\int_0^\pi \dfrac{2\vp - \sin(2\vp)}{2\vp^2} \cdot \left(\dfrac{\sin\vp}{\vp}\right)^{2n-1} \cdot \dfrac{\sin\vp -\vp\cos\vp}{\vp^2}~d\vp~.
\end{split}
\end{equation*}

Let $A$ be a compact subset of $(\Hn,d_c)$ with $\diam A = 2$. The projection $p$ defined as before by $p([z,t]) = z$ is 1-Lipschitz from $(\Hn,d_c)$ to  $(\Rnn,\|\cdot\|)$ hence \eqref{e:measproj} still holds for $A$. Next we recall that 
\begin{equation*}
 d_c([z,t],[z,t']) = (\pi|t-t'|)^{1/2}
\end{equation*}
hence $A\cap p^{-1}(\{z\})$ is contained in some segment with length at most $4/\pi$. It follows that 
\begin{equation*}
 \leb(A) \leq 4\pi^{-1}\alpha_{2n}.
\end{equation*}
All together this gives an upper bound for the isodiametric constant $C_{d_c}$ in $(\Hn,d_c)$,
\begin{equation*}
 C_{d_c} \leq \dfrac{1}{n\pi} \left(\int_0^\pi \dfrac{2\vp - \sin(2\vp)}{2\vp^2}\cdot\left(\dfrac{\sin\vp}{\vp}\right)^{2n-1}\cdot\dfrac{\sin\vp - \vp\cos\vp}{\vp^2}~d\vp~\right)^{-1}.
\end{equation*}
One can compute numerically this upper bound. For $n=1$ one gets an upper bound $\leq 1.22$. Then it increases for $n$ equals 1 to 8 up to a value $\leq 1.98$. For $n=9$ the upper bound is larger than 2. Thus we have the following 

\begin{thm} \label{prop:1/2bes-dc}
 Let $n\in\{1,\dots,8\}$. The isodiametric constant $C_{d_c}$ in $(\Hn,d_c)$ satisfies $C_{d_c}<2 $ and hence $\sigma_{2n+2}(\Hn,d_c)>1/2$.
\end{thm}

The previous argument leaves open the cases $n\geq 9$ and is probably far from being optimal. One could indeed probably expect that $C_{d_c}<2 $ holds in any dimension.



\end{document}